\documentclass[a4paper, reqno, 11pt]{amsart}

\usepackage{amscd, amssymb, amsthm, amsmath, color}
\usepackage{graphicx, psfrag, enumerate, multirow}

\newtheorem{lemma}{Lemma}
\newtheorem{theorem}[lemma]{Theorem}
\newtheorem{corollary}[lemma]{Corollary}
\newtheorem{proposition}[lemma]{Proposition}
\newtheorem{conjecture}[lemma]{Conjecture}
\newtheorem{remark}[lemma]{Remark}

\newcommand*\xbar[1]{%
   \hbox{%
     \vbox{%
       \hrule height 0.5pt 
       \kern0.5ex
       \hbox{%
         \ensuremath{#1}%
       }%
     }%
   }%
}

\begin{document}
\title{The Undirected Optical Indices of Complete $m$-ary Trees}

\author{Yuan-Hsun Lo}
\address{School of Mathematical Sciences, Xiamen University, Xiamen 361005, China}
\email[Y.-H.~Lo]{yhlo0830@gmail.com}

\author{Hung-Lin Fu}
\address{Department of Applied Mathematics, National Chiao Tung University, Hsinchu 300, Taiwan, ROC}
\email[H.-L.~Fu]{hlfu@math.nctu.edu.tw}

\author{Yijin Zhang}
\address{School of Electronic and Optical Engineering, Nanjing University of Science and Technology, Nanjing, China} 
\address{National Mobile Communications Research Laboratory, Southeast University, Nanjing, China}
\email[Y.~Zhang]{yijin.zhang@gmail.com}

\author{Wing Shing Wong}
\address{Department of Information Engineering, the Chinese University of Hong Kong, Shatin, Hong Kong}
\email[W.~S.~Wong]{wswong@ie.cuhk.edu.hk}



\thanks{This work was supported in part by the National Natural Science Foundation of China under grant numbers 61301107 and 11601454, the Natural Science Foundation of Fujian Province of China under grant number 2016J05021, the Fundamental Research Funds for the Central Universities in China under grant number 20720150210, the Ministry of Science and Technology, Taiwan under grant number 104-2115-M-009-009, and the open research fund of National Mobile Communications Research Laboratory, Southeast University, under grant number 2017D09.}


\maketitle

\begin{abstract}
The routing and wavelength assignment problem arises from the investigation of optimal wavelength allocation in an optical network that employs Wavelength Division Multiplexing (WDM).
Consider an optical network that is represented by a connected, simple graph $G$. 
An all-to-all routing $R$ in $G$ is a set of paths connecting all pairs of vertices of $G$.
The undirected optical index of $G$ is the minimum integer $k$ to guarantee the existence of a mapping $\phi:R\to\{1,2,\ldots,k\}$, such that $\phi(P)\neq\phi(P')$ if $P$ and $P'$ have common edge(s), over all possible routings $R$.
A natural lower bound of the undirected optical index of $G$ is the (undirected) edge-forwarding index, which is defined to be the minimum of the maximum edge-load over all possible all-to-all routings.
In this paper, we first derive the exact value of the optical index of the complete $m$-ary trees, and then investigate the gap between undirected optical and edge-forwarding indices.
\end{abstract}

\section{Introduction} \label{sec:intro}

Let $G=(V(G),E(G))$ be a connected simple graph with vertex set $V(G)$ and edge set $E(G)$.
An edge with endpoints $u$ and $v$ is represented as $\{u,v\}$.
A sequence of edges $P=(\{v_1,v_2\},\{v_2,v_3\},\ldots,\{v_\ell,v_{\ell+1}\})$ such that $\{v_i,v_{i+1}\}\in E(G)$ and $v_i\neq v_j$ for $i\neq j$ is called a \emph{path} of length $\ell$, with terminal vertices $v_1$ and $v_{\ell+1}$.
Two paths $P$ and $P'$ are said to be \emph{conflicting} if there exists an edge $e\in E(G)$ such that $e\in P$ and $e\in P'$.
A \emph{routing} in $G$ is a set of paths.
An \emph{all-to-all routing} $R_A$ in $G$ is a set of paths connecting all pairs of vertices of $G$.
Note that $|R_A|={|V(G)|\choose 2}$.
Throughout this paper, we always consider all-to-all routings.
Hence, we simply use routing to refer to an all-to-all routing and use $R$ for $R_A$ for the sake of convenience.

Let $\mathfrak{R}_G$ denote the collection of all routings in $G$.
For a given routing $R$, define the \emph{conflict graph} $Q(R)$ as the graph with vertex set $R$ by $P$ and $P'$ being adjacent if and only if they conflict.
The \emph{chromatic number} of a graph $G$ is the smallest number of colors needed to color the vertices of $G$ so that no two adjacent vertices share the same color.
Then, define the \emph{undirected optical index} of $G$ as $$w(G):=\min_{R\in\mathfrak{R}_G}\chi(Q(R)).$$

Analogous parameters can be introduced when considering directed graphs.
A \emph{symmetric directed graph} (\emph{digraph}) $G=(V(G),A(G))$ is a directed graph with vertex set $V(G)$ and arc set $A(G)$ such that if $(u,v)\in A(G)$ then $(v,u)\in A(G)$.
On the other hand, a symmetric digraph is obtained from a graph by putting two opposite arcs on each edge.
A directed path (dipath) with source $s$ and destination $d$ is a sequence of arcs $\vec{P}=((s=v_1,v_2),(v_2,v_3),\ldots,(v_\ell,v_{\ell+1}=d))$, for some $\ell>1$, such that $(v_i,v_{i+1})\in A(G)$ and $v_i\neq v_j$ for $v\neq j$.
Two dipaths $\vec{P}$ and $\vec{P'}$ are said to be conflicting if there exists an arc $\vec{e}\in A(G)$ such that $\vec{e}\in\vec{P}$ and $\vec{e}\in\vec{P'}$.
A directed (all-to-all) routing $\overrightarrow{R}$ in $G$ is a set of dipaths connecting all ordered pairs of vertices of $G$.
Note that $|\vec{R}|=|V(G)|\cdot (|V(G)|-1)$.
Let $\vec{\mathfrak{R}}_G$ denote the collection of all directed routings in $G$.
The \emph{directed optical index} of $G$ can be defined as $$\vec{w}(G):=\min_{\vec{R}\in\vec{\mathfrak{R}}_G}\chi(Q(\vec{R})),$$
where $Q(\vec{R})$ is the conflict graph induced by $\vec{R}$.
Obviously, $\vec{w}(G)\leq w(G)$ for any graph $G$.

The evaluation of $\vec{w}(G)$ is known as the \emph{routing and wavelength assignment} (RWA) problem, which aries from the investigation of optimal wavelength allocation in an optical network \cite{Alexander93,SS12} that employs Wavelength Division Multiplexing (WDM).
For $a\in A(G)$, the \emph{load} of $a$, denoted by $\ell_{\vec{R}}(a)$, is the number of dipaths passing through $a$ under the directed routing $\vec{R}$.
The \emph{arc-forwarding index} $\vec{\pi}(G)$ \cite{HMS89} is defined by
$$\vec{\pi}(G):=\min_{\vec{R}\in\vec{\mathfrak{R}}_G} \max_{a\in A(G)}\ell_{\vec{R}}(a).$$
Similarly, by letting $\ell_R(e)$ be the load of an edge $e\in E(G)$ under the routing $R$, the \emph{edge-forwarding index} $\pi(G)$ can be defined accordingly by
$$\pi(G):=\min_{R\in\mathfrak{R}_G} \max_{e\in E(G)}\ell_R(e).$$
Notice that $\pi(G)$ is in fact the minimum of the maximum size of cliques in $Q(R)$, over all possible $R\in\mathfrak{R}_G$, and $\vec{\pi}(\vec{R})$ is its directed version.
Since the number of colors needed in a graph is not less than the maximum size of its cliques, we have
\begin{equation}\label{eq:forwarding_bound}
w(G)\geq \pi(G) \quad \text{and} \quad \vec{w}(G)\geq \vec{\pi}(G).
\end{equation}

\begin{remark}\rm
It is worth noting that, in the literature \cite{ART01,Beauquier99,BGPRV00,GMZ15} the undirected optical index and the edge-forwarding index are usually defined on the routings which are obtained from directed routings by replacing each dipath with a path.
That is, each considered routing consists of two paths for every pairs of vertices.
In such a way, in order to avoid potential confusion, we call them \emph{double undirected optical index}, \emph{double edge-forwarding index}, \emph{double undirected routing}, and denote by $w^d$, $\pi^d$, $R^d$, respectively.
It is not hard to see that, for any simple graph $G$, $$\frac{1}{2}w^d\leq\vec{w}(G)\leq w(G) \quad \text{and} \quad \frac{1}{2}\pi^d\leq\vec{\pi}(G)\leq\pi(G).$$
\end{remark}

The study of directed optical index and arc-forwarding index has been intensive in literature.
It has been proved that the equality $\vec{w}(G)=\vec{\pi}(G)$ holds for cycles~\cite{BGPRV00,Wilfong96}, trees~\cite{GHP97}, trees of cycles~\cite{BPT99}, some Cartesian product of paths or cycles with equal lengths~\cite{Beauquier99,SSV97}, some certain compound graphs~\cite{ART01}.
However, $\vec{w}(G)$ may be strictly larger than $\vec{\pi}(G)$ in some cases \cite{Kosowski09}.
Pleaase refers to \cite{CQ09,NOS01,LZCFW17,SSV97,WLLCCZ16} for more information on directed optical index and \cite{QZ04,Sole95,XX13} on arc-forwarding index.
There are relatively few results on double undirected or undirected optical index.
It has been shown that $w^d(G)=\pi^d(G)$ for cycles and hypercubes~\cite{BGPRV00}.
In \cite{LZWF15}, the authors showed $w(C_n)=\pi(C_n)+1$ when $4|n$, and $w(C_n)=\pi(C_n)$ when $4\nmid n$, where $C_n$ refers to the cycle of size $n$.

In this paper we focus on the undirected optical index.
After introducing two useful lower bounds in Section~\ref{sec:preli}, we in Section~\ref{sec:main} derive the exact value of $w(T_{m,h})$ for any $m,h\geq 1$, where $T_{m,h}$ denotes a complete $m$-ary tree with height $h$.
The result reveals that in the case when $m>1$ is an odd integer, $w(T_{m,h})$ is strictly larger than $\pi(T_{m,h})$.
Finally, we investigate the ratio $w(G)/\pi(G)$ in Section~\ref{sec:w-pi}.

%
%

\section{Preliminaries}\label{sec:preli}
The edge-forwarding index is a natural lower bound of undirected optical index.
In this section we will further introduce a lower bound of the edge-forwarding index, which is clearly to be a lower bound of undirected optical index.
Given $S,T\subseteq V(G)$, let $[S,T]$ denote the set of edges having one endpoint in $S$ and the other in $T$.
An \emph{edge cut} is an edge set that can be represented as the form $[S,\xbar{S}]$, where $S$ is a non-empty proper subset of $V(G)$ and $\xbar{S}$ denotes $V(G)\setminus S$.
The \emph{Edge-cut Bound} for the edge-forwarding index is as follows.

\begin{proposition}[Edge-cut Bound]\label{pro:edge_cut}
For any simple graph $G$, we have 
\begin{equation}\label{eq:edge_cut}
\pi(G)\geq\max_{C=[S,\bar{S}]}\frac{|S|\cdot|\xbar{S}|}{|C|}.
\end{equation}
\end{proposition}
\begin{proof}
Assume $C=[S,\xbar{S}]$ is the edge cut that maximizes $\frac{|S|\cdot|\xbar{S}|}{|C|}$.
For any routing in $G$, there are $|S|\cdot|\xbar{S}|$ paths having one endpoint in $S$ and the other in $\xbar{S}$.
Hence the result follows by the fact that each of those paths uses at least one edge in $C$.
\end{proof}

Analogous to the Edge-cut Bound for the edge-forwarding index, we introduce the Vertex-cut Bound for the undirected optical index.
For a vertex set $S\subseteq V(G)$, denote by $G-S$ the graph induced by $V(G)\setminus S$.
A \emph{vertex cut} of a graph $G$ is a set $S\subseteq V(G)$ such that $G-S$ has more than one component.
Denote by $H_1^S,H_2^S,\ldots$ the components in $G-S$.
The \emph{Vertex-cut Bound} for the optical index is as follows.

\begin{proposition}[Vertex-cut Bound]\label{pro:vertex_cut}
For any simple graph $G$, we have
\begin{equation}
w(G)\geq\max_{S\subseteq V(G)}\frac{\sum_{i\neq j}|H_i^S|\cdot|H_j^S|}{\big\lfloor{|[S,\bar{S}]|}\big/{2}\big\rfloor}.
\label{eq:vertex_cut}
\end{equation}
\end{proposition}
\begin{proof}
Let $R$ be a routing in $G$, and let $S$ be a vertex cut.
There are $\sum_{i\neq j}|H_i^S|\cdot|H_j^S|$ paths in $R$ whose endpoints are in distinct components in $G-S$.
Since each of such paths uses at least two edges in $[S,\xbar{S}]$, at most $\lfloor[S,\xbar{S}]\big/2\rfloor$ paths can receive the same color.
Then we have $$\chi(Q(R))\geq \frac{\sum_{i\neq j}|H_i^S|\cdot|H_j^S|}{\big\lfloor{|[S,\bar{S}]|}\big/{2}\big\rfloor},$$ and hence the result follows.
\end{proof}

\section{Main Results}\label{sec:main}
For $m,h\geq 1$, let $T_{m,h}$ be a complete $m$-ary tree with height $h$.
Let $r$ denote the root of $T_{m,h}$, $r_1,r_2,\ldots,r_m$ denote the children of $r$, $r_{i,1}, r_{i,2},\ldots, r_{i,m}$ denote the children of $r_i$ for $1\leq i\leq m$, $r_{i,j,1}, r_{i,j,2},\ldots, r_{i,j,m}$ denote the children of $r_{i,j}$ for $1\leq i,j\leq m$, and so on.
More precisely, the vertex lying on the $k$-th layer is denoted by $r_w$, where $w=t_1,t_2,\ldots,t_k$ is a $k$-tuple with $1\leq t_i\leq m$, and $r_{w'}$ is a child of $r_w$ if and only if $w'=w,t_{k+1}$ for some $t_{k+1}\in\{1,2,\ldots,m\}$.
Let $e_w$ denote the edge connecting the vertex $r_w$ with its parent.
Moreover, we use $H_w$ to denote the subtree rooted by $r_w$.
In particular, $H=T_{m,h}$.
See Fig.~\ref{fig:T33} for the example of $T_{3,3}$.

\begin{figure}[h]
\centering
\includegraphics[width=4.5in]{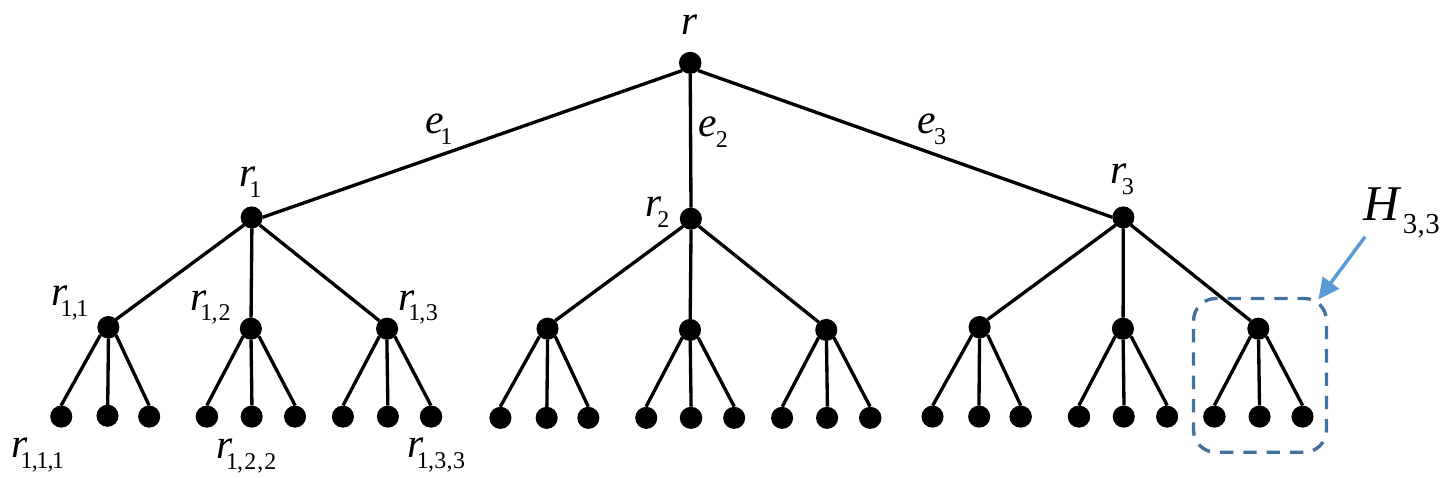}
\caption{$T_{3,3}$.} \label{fig:T33}
\end{figure}

The main result of this paper is to give the explicit value of optical index of $T_{m,h}$, for any positive integers $m$ and $h$.

\begin{theorem}\label{thm:m-ary}
We have
\begin{enumerate}[(i)]
\item $w(T_{1,h})=\lfloor\frac{h+1}{2}\rfloor\lceil\frac{h+1}{2}\rceil$;
\item $w(T_{2,1})=2$, $w(T_{2,2})=12$, and $w(T_{2,h})=5\cdot 2^{2h-2}-3\cdot 2^h+1$ for $h\geq 3$;
\item If $m$ is odd and $m\geq 3$, $w(T_{m,h})=m(1+m+m^2+\cdots+m^{h-1})^2$; and
\item If $m$ is even and $m\geq 4$, $w(T_{m,h})=m^h(1+m+m^2+\cdots+m^{h-1})$.
\end{enumerate}
\end{theorem}

Since $T_{1,h}$ is identical to a path of length $h$, the conflict graph $Q(R)$ of the unique routing $R$ can be viewed as the interval graph of intervals $(i,j)$, where $1\leq i<j\leq h+1$ and $i,j\in\mathbb{N}$.
By the well known result that the chromatic number of an interval graph is equal to the maximum size of its cliques, we have $w(T_{1,h})=\pi(T_{1,h})=\lfloor\frac{h+1}{2}\rfloor\lceil\frac{h+1}{2}\rceil$.

The proof of the rest three cases of Theorem~\ref{thm:m-ary} is put in the following subsections.
If there is no danger of confusion, we assume $R$ denote the unique routing of $T_{m,h}$, as there is only one path that connects any pair of vertices.

\subsection{Proof of Theorem~\ref{thm:m-ary}(ii)}

First, we consider the lower bound of $w(T_{2,h})$ by both Edge-cut Bound and Vertex-cut Bound.

Consider the edge cut $\{e_1\}$.
By the Edge-cut Bound we have
\begin{equation}\label{eq:T_2-1}
w(T_{2,h})\geq \pi(T_{2,h}) \geq 2^h(1+2+2^2+\cdots+2^{h-1}) = 2^{2h}-2^h.
\end{equation}
Consider the vertex cut $S=\{r_1\}$.
Then $T_{2,h}-S$ has three components: $H_{1,1},H_{1,2}$ and $r+H_2$.
Since $|H_{1,1}|=|H_{1,2}|=2^{h-1}-1$ and $|r+H_2|=2^h$, by the Vertex-cut Bound we have 
\begin{align}
w(T_{2,h}) &\geq \frac{|H_{1,1}|\cdot|H_{1,2}|+|H_{1,1}|\cdot|r+H_2|+|H_{1,2}|\cdot|r+H_2|}{\lfloor 3/2\rfloor} \notag \\
&= 5\cdot 2^{2h-2}-3\cdot 2^h+1. \label{eq:T_2-2}
\end{align}
Since $5\cdot 2^{2h-2}-3\cdot 2^h+1>2^{2h}-2^h$ for $h\geq 3$, combining \eqref{eq:T_2-1} and \eqref{eq:T_2-2} yields
\begin{equation} \label{eq:T_2-3}
w(T_{2,h})\geq \begin{cases}
2 & \text{if } h=1; \\
12 & \text{if } h=2; \text{ and}\\
5\cdot 2^{2h-2}-3\cdot 2^h+1, & \text{if } h\geq 3.
\end{cases}
\end{equation}

Now, we shall inductively define a coloring $\phi$ on $Q(R)$ by using the number of colors as shown in \eqref{eq:T_2-3}.
As the colorings for $h=1$ and $h=2$ are shown in Table~\ref{table:T_2-1} and Table~\ref{table:T_2-2}, respectively, in what follows we consider $h\geq 3$.

\begin{table}[h]
\begin{tabular}{c|l}
\hline
color & path(s) \\ \hline
1 & $\{r,r_1\}$, $\{r,r_2\}$ \\
2 & $\{r_1,r_2\}$ \\ \hline
\end{tabular}
\caption{$2$-coloring on $Q(R)$ for $T_{2,1}$.}
\label{table:T_2-1}
\end{table}

\begin{table}[h]
\begin{tabular}{c|l||c|l}
\hline
color & path(s) & symbol & path(s) \\ \hline
1 & $\{r,r_1\}$, $\{r,r_2\}$ & 7 & $\{r_{1,1},r_2\}$ \\
2 & $\{r,r_{1,1}\}$, $\{r,r_{2,1}\}$ & 8 & $\{r_{1,1},r_{2,1}\}$ \\ 
3 & $\{r,r_{1,2}\}$, $\{r,r_{2,2}\}$ & 9 & $\{r_{1,1},r_{2,2}\}$ \\ 
4 & $\{r_1,r_2\}$, $\{r_{1,1},r_{1,2}\}$, $\{r_{2,1},r_{2,2}\}$ & 10 & $\{r_{1,2},r_2\}$ \\ 
5 & $\{r_1,r_{2,1}\}$, $\{r_1,r_{1,2}\}$, $\{r_2,r_{2,2}\}$ & 11 & $\{r_{1,2},r_{2,1}\}$ \\ 
6 & $\{r_1,r_{2,2}\}$, $\{r_1,r_{1,1}\}$, $\{r_2,r_{2,1}\}$ & 12 & $\{r_{1,1},r_{2,2}\}$ \\ \hline
\end{tabular}
\caption{$12$-coloring on $Q(R)$ for $T_{2,2}$.}
\label{table:T_2-2}
\end{table}

By \eqref{eq:T_2-1}, there are $2^{2h}-2^h$ paths containing the edge $e_1$.
We partition these paths into the following $9$ classes.
See Fig.~\ref{fig:T2h} for the structure of $T_{2,h}$.

\begin{figure}[h]
\centering
\includegraphics[width=3in]{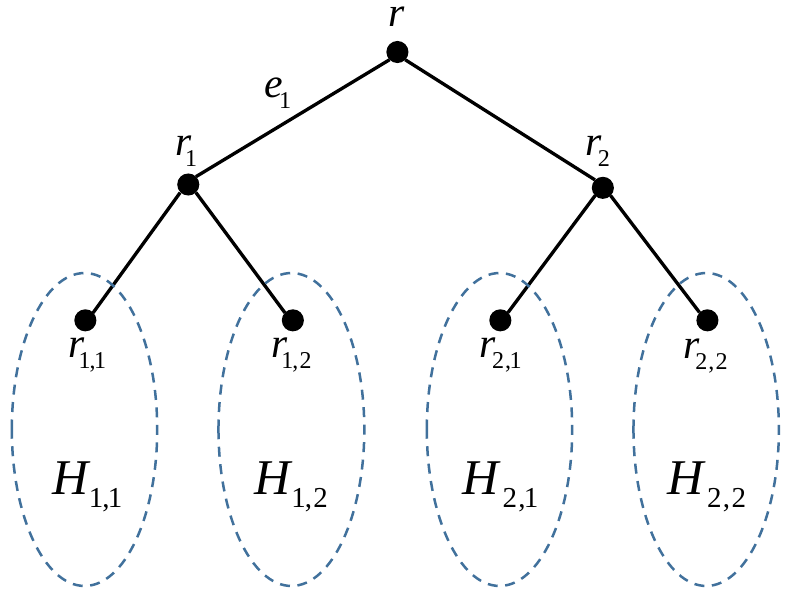}
\caption{An illustration of $T_{2,h}$.} \label{fig:T2h}
\end{figure}

\begin{enumerate}[{Class} 1.]
\item The path connecting $r$ and $r_1$, i.e., the edge $e_1$.
\item The path connecting $r_1$ and $r_2$.
\item The path connecting $r$ and one vertex in $H_{1,1}$ or $H_{1,2}$.
\item The path connecting $r_1$ and one vertex in $H_{2,1}$ or $H_{2,2}$.
\item The path connecting $r_2$ and one vertex in $H_{1,1}$ or $H_{1,2}$.
\item The path having one terminal vertex in $H_{1,1}$ and the other in $H_{2,1}$.
\item The path having one terminal vertex in $H_{1,1}$ and the other in $H_{2,2}$.
\item The path having one terminal vertex in $H_{1,2}$ and the other in $H_{2,1}$.
\item The path having one terminal vertex in $H_{1,2}$ and the other in $H_{2,2}$.
\end{enumerate}

We denote by $A_i$, $i=1,2,\ldots,9$, the set of paths in Class $i$.
One can check that $|\biguplus_{i=1}^9 A_i|=2^{2h}-2^h$.
The objective coloring is as follows.

\noindent
{\bf Canonical coloring for $T_{2,h}$.}
\begin{enumerate}[Step 1.]
\item Each path in $A_i$, $i=1,2,\ldots,9$, receives a distinct color. Denote the color sets by $C_1,C_2,\ldots,C_9$, respectively.
\item Assign the colors in $\biguplus C_i$ to as many remaining paths as possible. 
In other words, for each $i\in\{1,2,\ldots,9\}$, pick as many mutually edge-disjoint paths which are not included in $A_i$ as possible.
\item Each of the remaining paths after Step 2 receives a new color.
\end{enumerate}
Table~\ref{table:T_2-h} shows how Step 2 above is accomplished.
For convenience, denote by $\mathcal{P}_1$ ($\mathcal{P}_2$, resp.) be the set of paths having one terminal vertex in $H_{1,1}$ ($H_{2,1}$, resp.) and the other in $H_{1,2}$ ($H_{2,2}$, resp.).
Note that paths in $\mathcal{P}_1$ should receive distinct colors, since each of them contains edges $e_{1,1}$ and $e_{1,2}$.
Similarly, all paths in $\mathcal{P}_2$ should receive distinct colors.

\begin{table}[h]
\begin{tabular}{c|l|l}
\hline
$i$ & $|C_i|$ & Chosen edge-disjoint paths not included in $A_i$ \\ \hline
1 & $1$ & $\{r,r_2\}$, one path from $\mathcal{P}_1$ and one from $\mathcal{P}_2$ \\
2 & $1$ & one path from $\mathcal{P}_1$ and one from $\mathcal{P}_2$ \\ 
3 & $2^h-2$ & $\{\{r,r_w\}:\,r_w\text{ is in } H_{2,1}\text{ or }H_{2,2}\}$ \\
4 & $2^h-2$ & any $2^h-2$ paths from $\mathcal{P}_1$ \\ 
5 & $2^h-2$ & any $2^h-2$ paths from $\mathcal{P}_2$ \\ 
6 & $2^{2h-2}-2^h+1$ & paths with terminal vertices in $r_1+H_{1,2}$ \\ 
7 & $2^{2h-2}-2^h+1$ & paths with terminal vertices in $r_2+H_{2,1}$ \\ 
8 & $2^{2h-2}-2^h+1$ & paths with terminal vertices in $r_2+H_{2,2}$ \\ 
9 & $2^{2h-2}-2^h+1$ & paths with terminal vertices in $r_1+H_{1,1}$ \\ \hline
\end{tabular}
\caption{Edge-disjoint paths added to each class.}
\label{table:T_2-h}
\end{table}

Following Table~\ref{table:T_2-h}, as the number of edge-disjoint paths added for $A_i$, $1\leq i\leq 5$, is equal to $|C_i|$, we only need to claim that, for $i=6,7,8,9$, the number $|C_i|=2^{2h-2}-2^h+1$ of colors is enough for the extra paths we added for $A_i$.
Without loss of generality, we consider the case that $A_6$.
Since $H_{1,2}$ is isomorphic to $T_{2,h-2}$ and $|V(T_{2,h-2})|=2^{h-1}-1$, it suffices to show $$(2^{2h-2}-2^h+1) - w(T_{2,h-2}) \geq 2^{h-1}-1.$$
This inequality always holds for $h\geq 3$ due to $w(T_{2,1})=2$, $w(T_{2,2})=12$ and the induction hypothesis that $w(T_{2,h-2})=5\cdot 2^{2h-6}-3\cdot 2^{h-2}+1$ for $h\geq 5$.

After the Step 2 of the Canonical Coloring for $T_{2,h}$, all the paths left are either in $\mathcal{P}_1$ or $\mathcal{P}_2$.
By ignoring those paths that have received colors in Step 2, there are
\begin{equation}\label{eq:T_2h-step3}
(2^{h-1}-1)^2-2^h=2^{2h-2}-2\cdot 2^{h}+1
\end{equation}
paths in $\mathcal{P}_1$ (and $\mathcal{P}_2$) left in Step 3.
Since Step 1 and Step 2 cost 
\begin{equation}\label{eq:T_2h-step12}
\sum_{i=1}^9|C_i|=4\cdot 2^{2h-2}-2^h
\end{equation}
colors, combining \eqref{eq:T_2h-step3} and \eqref{eq:T_2h-step12} implies that the Canonical Coloring for $T_{2,h}$ valid by using $5\cdot 2^{2h-2}-3\cdot 2^h+1$ colors.
This completes the proof.
\qed

\subsection{Proof of Theorem~\ref{thm:m-ary}(iii)}

In this part, we further consider a more general case.
For any two positive integers $k$ and $t$, let $G_{k,t}$ be a simple graph with a cut vertex, say $r$, such that $G_{k,t}-r$ has $k$ components, each of which contains $t$ vertices and connects to $r$ by one edge.
In other words, the root $r$ is of degree $k$ in $G_{k,t}$.
$T_{m,h}$ turns out to be a special case of $G_{k,t}$ by plugging $k=m$ and $t=\sum_{i=0}^{h-1}m^i$.
To prove Theorem~\ref{thm:m-ary}(iii), therefore, it suffices to show the following.

\begin{lemma}
For any odd integer $k\geq 3$ and any positive integer $t$, we have $$w(G_{k,t})=kt^2.$$
\label{lem:G_kt}
\end{lemma}
\noindent\emph{Proof of Lemma~\ref{lem:G_kt}.}
Denote by $G_1,G_2,\ldots,G_k$ the $k$ components of $G_{k,t}-r$.
So we have $|V(G_1)|=|V(G_2)|=\cdots=|V(G_k)|=t$.
Consider the vertex cut $S=\{r\}$. 
Since there is only one edge connecting $r$ and each of the $k$ components, $|[S,\bar{S}]|=k$.
By the assumption that $k$ is odd, it follows from the Vertex-cut Bound that
\begin{equation}
w(G_{k,t})\geq \frac{{k\choose 2}t^2}{(k-1)/2} = kt^2.
\label{eq:G_kt_lower}
\end{equation}

To prove $w(G_{k,t})\leq kt^2$, we revisit the concept of total coloring of a graph.
A \emph{$k$-total-coloring} of a graph $G$ is a mapping from $V(G)\cup E(G)$ into a set of colors $\{1,2,\ldots,k\}$ such that (i) adjacent vertices in $G$ receive distinct colors, (ii) incident edges in $G$ receive distinct colors, and (iii) any vertex and its incident edges receive distinct colors.
The \emph{total chromatic number} $\chi''(G)$ of a graph $G$ is the minimum number $k$ for which $G$ has a $k$-total-coloring.
The following result is known.

\begin{lemma}[\cite{Yap96}, p.16]
If $n$ is an odd positive integer, then $\chi''(K_n)=n$.
\label{lem:total-coloring}
\end{lemma}

Let $R$ to be an arbitrary routing in $G_{k,t}$.
In what follows, we will define a $(kt^2)$-coloring $\phi$ on $Q(R)$.
Since $k$ is an odd positive integer, by Lemma~\ref{lem:total-coloring}, we can find $f$ to be a $k$-total-coloring on a complete graph of vertex set $\{v_1,v_2,\ldots,v_k\}$ with color set $\{1,2,\ldots,k\}$.
Let $e_{\{i,j\}}$ denote the edge with endpoints $v_i$ and $v_j$.
For $1\leq i\leq k$ let $\mathcal{P}_i\subset R$ denote the set of paths in $R$ whose terminal vertices are both in $V(G_i)\cup\{r\}$; and, for $1\leq i<j\leq k$, let $\mathcal{P}_{(i,j)}\subset R$ denote the set of paths in $R$ having one terminal vertex in $V(G_i)$ and the other in $V(G_j)$.
It is easy to see that $$R=\biguplus_{1\leq i\leq k}\mathcal{P}_i\cup\biguplus_{1\leq i<j\leq k}\mathcal{P}_{(i,j)}.$$ 
Let $C=C_1\uplus C_2\uplus\cdots\uplus C_k$ be a collection of $kt^2$ colors, where each $C_i$ contains exactly $t^2$ of them.
Then, the objective coloring $\phi$ of $Q(R)$ is defined as follows.
\begin{enumerate}[(i)]
\item For $1\leq i\leq k$, each path in $\mathcal{P}_i$ receives a distinct color from $C_{f(v_i)}$; and
\item For $1\leq i< j\leq k$, each path in $\mathcal{P}_{(i,j)}$ receives a distinct color from $C_{f(e_{\{i,j\}})}$.
\end{enumerate}
$\phi$ is well-defined because $|\mathcal{P}_{(i,j)}|=t^2$ and  $|\mathcal{P}_i|={t+1\choose 2}\leq t^2$ due to $k\geq 1$.
Moreover, since $f$ is a $k$-total-coloring, any two paths in $R$ receive the same color under $\phi$ have no common edge(s).
Hence $\phi$ is a $(kt^2)$-coloring of $Q(R)$, and thus
\begin{equation}
w(G_{k,t})\leq kt^2.
\label{eq:G_kt_upper}
\end{equation}
The result follows by combining \eqref{eq:G_kt_lower} and \eqref{eq:G_kt_upper}.
\qed



\subsection{Proof of Theorem~\ref{thm:m-ary}(iv)}

In this case, $m$ is even and $m\geq 4$.
The lower bound of $w(T_{m,h})$ can be obtained directly from the Edge-cut Bound by choosing $\{e_1\}$ to be the edge cut.
In what follows, we will inductively define a coloring $\phi$ on $Q(R)$ using $m^h(1+m+m^2+\ldots+m^{h-1})$ colors.

Let $D_{m,h}$ be a tree obtained from two $T_{m,h-1}$'s by connecting the two roots, and let $A$ denote the left copy of $T_{m,h-1}$ and $B$ denote the right one.
Denote the vertices of subtrees $A$ and $B$ in the same way as in $T_{m,h-1}$ but replacing letter $r$ by $a$ and $b$, respectively.
Also, denote by $A_w$ ($B_w$, resp.) the subtree rooted by $a_w$ ($b_w$, resp.) for any $w$.
See Fig.~\ref{fig:D42} for an example of $D_{4,2}$.

\begin{figure}[h]
\centering
\includegraphics[width=3in]{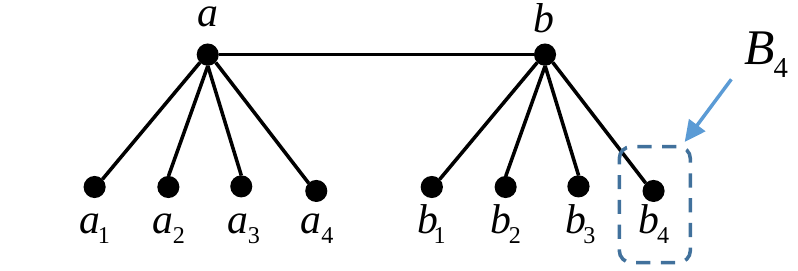}
\caption{An illustration of $D_{4,2}$.} \label{fig:D42}
\end{figure}

In what follows, for any even number $m\geq 4$, we will show both
\begin{equation}\label{eq:D_mh}
w(D_{m,h})\leq (1+m+m^2+\cdots+m^{h-1})^2
\end{equation}
and
\begin{equation}\label{eq:T_4mh}
w(T_{m,h})\leq m^h(1+m+m^2+\ldots+m^{h-1})
\end{equation}
by induction on $h$.
We refer $R^D$ to the unique routing in $D_{m,h}$, as $R$ refers to that in $T_{m,h}$.

As it is obviously that \eqref{eq:D_mh} holds for $h=1$, we consider $w(T_{m,1})$.
Recall that $K_m$ is $(m-1)$-edge-colorable.
Here, a $k$-edge-coloring of a graph $G$ is a mapping from $E(G)$ into a set of colors $\{1,2,\ldots,k\}$ such that incident edges in $C$ receive distinct colors.
Let $f$ be an $(m-1)$-edge-coloring on a complete graph of vertex set $\{v_1,v_2,\ldots,v_m\}$ with color set $\{1,2,\ldots,m-1\}$.
Let $e_{i,j}$ denote the edge with endpoints $v_i$ and $v_j$.
Then, the coloring $\phi$ of $Q(R)$ for $T_{m,1}$ can be given as follows.
\begin{itemize}
\item For $1\leq i\leq m$, let $\phi(P_{\{r,r_i\}})=m$.
\item For $1\leq i,j\leq m$, $i\neq j$, let $\phi(P_{\{r_i,r_j\}})=f(e_{i,j})$.
\end{itemize}
Hence \eqref{eq:T_4mh} holds for $h=1$.

\medskip

When $h\geq 2$, we first consider $D_{m,h}$.
Partition $R^D$ into the following two parts.
\begin{enumerate}[(i)]
\item $\mathcal{P}_1:=$ the collection of paths having terminal vertices in $V(D_{m,h})\setminus\{a,b\}$; and
\item $\mathcal{P}_2:=$ the collection of paths having at least one terminal vertex is $a$ or $b$.
\end{enumerate} 
Fix $1\leq i,j\leq m$.
Consider the set of paths having one terminal vertex in $V(A_i)$ and the other in $V(B_j)$, and denote it by $\mathcal{P}_{i,j}$.
Since each path in $\mathcal{P}_{i,j}$ contains the edge $\{a,b\}$, it needs $|\mathcal{P}_{i,j}|=(1+m+m^2+\cdots+m^{h-2})^2$ colors for $\mathcal{P}_{i,j}$.
Let $C_{i,j}$ be the set of those colors.
Observe that, $C_{i,j}$ can be used to color all the paths having one terminal vertex in $V(A_{i_1})$ (or $V(B_{j_1})$) and the other in $V(A_{i_2})$ (or $V(B_{j_2})$), whenever $i\notin\{i_1,i_2\}$ (or $j\notin\{j_1,j_2\}$).
Moreover, $C_{i,j}$ can also be used to color all the paths having both terminal vertices in $V(A_{i'})$ (or $V(B_{j'})$) if $i'\neq i$ (or $j'\neq j$), since in $A_{i'}$ (or in $B_{j'}$) a routing contains $${1+m+m^2+\cdots+m^{h-2}\choose 2} \leq (1+m+m^2+\cdots+m^{h-2})^2=|C_{i,j}|$$ paths.
By going through all $1\leq i,j\leq m$, it is not hard to verify that all paths in $\mathcal{P}_1$ can be colored by $m^2(1+m+m^2+\cdots+m^{h-2})^2$ colors; that is,
\begin{equation}
\chi(Q(\mathcal{P}_1)) \leq m^2(1+m+m^2+\cdots+m^{h-2})^2.
\label{eq:D_mh-1}
\end{equation}
When it comes to $\mathcal{P}_2$, the path $P_{a,a_w}$ can receive the same color as $P_{a,b_w}$, while the path $P_{b,b_w}$ can receive the same color as $P_{b,a_w}$.
So, we have
\begin{equation}
\chi(Q(\mathcal{P}_2)) \leq 2m(1+m+m^2+\cdots+m^{h-2})+1,
\label{eq:D_mh-2}
\end{equation}
which is equal to the number of paths in $\mathcal{P}_2$ passing through the edge $\{a,b\}$.
Hence \eqref{eq:D_mh} can be derived by combining \eqref{eq:D_mh-1} and \eqref{eq:D_mh-2}.


Finally, we consider $T_{m,h}$.
For $1\leq i\leq m$ let $\mathcal{P}_i$ denote the set of paths whose terminal vertices are both in $V(H_i)$; and, for $1\leq i\neq j\leq m$ let $\mathcal{P}_{\{i,j\}}$ denote the set of paths having one terminal vertex in $V(H_i)$ and the other in $V(H_j)$.
Note that $$\mathcal{P}=\mathcal{H}\cup\biguplus_{1\leq i\leq m}\mathcal{P}_i\cup\biguplus_{1\leq i\neq j\leq m}\mathcal{P}_{\{i,j\}},$$ where $\mathcal{H}$ denote the paths having $r$ as a terminal vertex.
Let $C=C_1\uplus C_2\uplus\cdots\uplus C_{m-1}$ be a collection of mutually disjoint color sets, where each $C_i$ contains exactly $(1+m+m^2+\cdots+m^{h-1})^2$ colors.
Recall that $f$ is an $(m-1)$-edge-coloring on $\{v_1,v_2,\ldots,v_m\}$ with color set $\{1,2,\ldots,m-1\}$.
Then, for pairs $\{i,j\}$ satisfying $f(e_{\{i,j\}})=1$, the paths in $\mathcal{P}_{\{i,j\}}\cup\mathcal{P}_i\cup\mathcal{P}_j$ can be colored by the colors in $C_1$.
This can be done because $\mathcal{P}_{\{i,j\}}\cup\mathcal{P}_i\cup\mathcal{P}_j$ is identical to $R^D$, the routing defined in $D_{m,h}$, whose optical index is not larger than $|C_1|$ by \eqref{eq:D_mh}.
For $k=2,3,\ldots,m-1$, if $f(e_{\{i,j\}})=k$, then color the paths in $\mathcal{P}_{\{i,j\}}$ by the colors in $C_k$.
Finally, by coloring paths $P_{\{r,r_{i,w}\}}$ and $P_{\{r,r_{j,w}\}}$ with the same color for $i\neq j$, we can use $1+m+m^2+\ldots+m^{h-1}$ colors for the coloring of paths in $\mathcal{H}$. 
Hence we have 
\begin{align*}
w(T_{m,h}) &\leq (m-1)(1+m+m^2+\ldots+m^{h-1})^2 + (1+m+m^2+\ldots+m^{h-1}) \\
&= m^h(1+m+m^2+\ldots+m^{h-1}),
\end{align*}
as desired.
\qed

\section{The gap between $w$ and $\pi$}\label{sec:w-pi}

Recall that $G_{k,t}$ is a simple graph with $r$ as a cut vertex such that $G_{k,t}-r$ has $k$ components having the same number $t$ of vertices.
In this section we restrict the graph $G_{k,t}$ to a be a tree, and use the notation $G^T_{k,t}$ to emphasize the tree structure. 
The $k$ components in $G^T_{k,t}-r$ are denoted by $G^T_1, G^T_2,\ldots, G^T_k$ accordingly.

\begin{proposition}
Let $k\geq 2$ and $t$ be two positive integers.
We have $$\pi(G^T_{k,t})=(k-1)t^2+t.$$
\label{pro:GT_kt}
\end{proposition}
\proof
For $i=1,2,\ldots,k$ let $e_i$ denote the edge connecting $r$ and $G_i$.
By choosing the edge cut $C=\{e_i\}$, for any $i$, following the Edge-cut Bound we have
\begin{equation}
\pi(G^T_{k,t})\geq (k-1)t^2+t.
\label{eq:GT_kt-lower}
\end{equation}
Assume to the contradiction that $\pi(G^T_{k,t})>(k-1)t^2+t$, that is, there exists an edge, say $e$, passed by more than $(k-1)t^2+t$ paths in the unique routing.
By the Edge-cut Bound and \eqref{eq:GT_kt-lower}, $e$ must be in $E(G_i)$ for some $i$.
Let $A$ be the component of $G^T_{k,t}-e$ which does not contain $r$, and denote by $a$ the number of vertices in $A$.
Then, there are $(kt+1-a)a$ paths containing the edge $e$.
Note here that $1\leq a\leq t-1$.
Define a function $f(x):=(kt+1-x)x$ on the interval $[1,t-1]$.
Since $k\geq 2$, it follows that $$f'(x)=-2x+kt+1>0, \quad \text{for }x\in[1,t-1].$$
This implies the maximum number of paths containing the edge $e$ should be $$f(t-1)=(k-1)t^2+t-(k-2)t-2,$$
which is less than $(k-1)t^2+t$ for $k\geq 2$ and $t\geq 1$.
This is a contradiction to the assumption.
\qed

By Lemma~\ref{lem:G_kt} and Proposition~\ref{pro:GT_kt}, in the case when $k\geq 3$ is odd, we have
\begin{equation}
\frac{w(G^T_{k,t})}{\pi(G^T_{k,t})}=\frac{kt}{(k-1)t+1}.
\label{eq:ratio_GT_kt}
\end{equation}
When $t=1$, the ratio in \eqref{eq:ratio_GT_kt} is equal to $1$; while, as $t\to\infty$, this ratio approaches to $\frac{k}{k-1}$, whose maximal value is $\frac{3}{2}$ due to $k\geq 3$.
Let $\delta\geq 1$ be a rational number.
We call $\delta$ is \emph{feasible} if there exists a simple graph $G$ such that $$\frac{w(G)}{\pi(G)}=\delta.$$

\begin{corollary}\label{cor:feasible}
We have
\begin{enumerate}[(i)]
\item for any integer $h>0$, $\delta=1+\frac{2^{2h-2}- 2^{h+1}+1}{2^{2h}-2^h}$ is feasible; and
\item for any integers $k,t>0$, $\delta=1+\frac{t-1}{(k-1)t+1}$ is feasible.
\end{enumerate}
\end{corollary}
\proof
By plugging $k=2$ and $t=2^h-1$ into Proposition~\ref{pro:GT_kt}, we have $\pi(T_{2,h})=2^{2h}-2^h$.
Then, the first result follows by combining it with Theorem~\ref{thm:m-ary}(ii).
The second result can be obtained directly by \eqref{eq:ratio_GT_kt}, or Lemma~\ref{lem:G_kt} and Proposition~\ref{pro:GT_kt}.
\qed

We end this paper by the following two conjectures.

\begin{conjecture}
For any tree $T$, $1\leq w(T)/\pi(T)\leq\frac{3}{2}$.
\end{conjecture}

\begin{conjecture}
Any rational number $\delta$ satisfying $1\leq\delta\leq\frac{3}{2}$ is feasible.
\end{conjecture}
\medskip

\noindent\textbf{Acknowledgements.}
The authors would like to express their gratitude to the referee for his or her valuable comments and suggestions in improving the presentation of this paper.


\rm

\end{document}